\documentclass[english,11pt]{article}

\usepackage[english]{babel}
\usepackage{amsmath,amsthm,amssymb,enumerate}
\usepackage[latin1]{inputenc}
\usepackage[T1]{fontenc}
\usepackage{psfig,labelfig}
\usepackage{graphicx}
\usepackage[dvips]{epsfig}

\newtheorem{thm}{Theorem}[section]

\newtheorem{lem}[thm]{Lemma}
\newtheorem{nte}[thm]{Note}

\newcommand{\N}{{\mathbb N}}

\DeclareMathOperator{\arcsinh}{arcsinh}

\DeclareMathOperator{\sy}{sys}
\DeclareMathOperator{\msy}{msys}
\DeclareMathOperator{\dist}{dist}

\title{Construction of hyperbolic Riemann surfaces with large systoles}

\author{Hugo Akrout and Bjoern Muetzel}

\begin{document}

\maketitle

\begin{abstract}
Let $S$ be a compact hyperbolic Riemann surface of genus $g \geq 2$. We call a \textit{systole} a shortest simple closed geodesic in $S$ and denote by $\sy(S)$ its length. Let $\msy(g)$ be the maximal value that $\sy(\cdot)$ can attain among the compact Riemann surfaces of genus $g$. We call a \textit{(globally) maximal} surface $S_{max}$ a compact Riemann surface of genus $g$ whose systole has length $\msy(g)$. In Section 2 we use cutting and pasting techniques to construct compact hyperbolic Riemann surfaces with large systoles from maximal surfaces. This enables us to prove several inequalities relating $\msy(\cdot)$ of different genera. In Section 3 we derive similar intersystolic inequalities for non-compact hyperbolic Riemann surfaces with cusps.\\  

\end{abstract}

\section{Introduction}
Let $S$ be a compact hyperbolic Riemann surface of genus $g \geq 2$. A systole of $S$ is a shortest simple closed geodesic. We denote by $\sy(S)$ its length. Let $\msy(g)$ be the value
\[
     \msy(g) = \sup \{ \sy(S) \mid S \text{ compact hyperbolic Riemann surface of genus } g \geq 2 \}.
\]
Due to Mumford's generalization of Mahler's compactness theorem in \cite{mu}, this supremum is a maximum. The exact value of $\msy(g)$ is only known for $g=2$. We also have the following estimates from \cite{bs}:\\
There exists a universal, but unknown, constant $C > 0$ such that for all genera $g \geq 2$
\begin{equation}
      C \cdot \log(g) \leq \msy(g) \leq 2 \log(4g-2).
\label{eq:log_bound}
\end{equation}
Here the upper bound follows from a simple area argument (see \cite{bs}). The first account of a lower bound is due to Buser in \cite{bu1}. Here an infinite sequence of surfaces with lower bound of order $\sqrt{\log(g)}$ is constructed. Using arithmetic surfaces it was then shown by Buser and Sarnak in \cite{bs} that there is a infinite sequence of genera $(g_k)_k$  with
\[
    \msy(g_k) \ge \frac{4}{3} \log(g_k)-c_0
\]
where $c_0$ is constant. This construction was then generalized in \cite{ksv1}, \cite{ksv2} and \cite{kksv}, where more families of compact hyperbolic Riemann surfaces satisfying the above inequality can be found. Asymptotically the factor $\frac{4}{3}$ is the best known to date and Makisumi (see \cite{ma}, \textbf{Theorem 1.6}) showed that, in some sense, this lower bound is optimal for the generalized Buser-Sarnak construction.\\
The study of surfaces whose systole length is a global or local maximum in the moduli space $\mathcal{M}_g$ of compact hyperbolic Riemann surfaces of genus $g \geq 2$ was initiated by Schmutz (see \cite{sc2}, \cite{sc1}  and \cite{sc4}). Here he also provides a number of interesting properties of these surfaces. The characterization of maximal surfaces was continued in  \cite{ba}, \cite{ak}, \cite{ge} and \cite{pa1}. Here it was shown that
\begin{itemize}
\item A (locally) maximal surface of genus $g$ has at least $6g-5$ systoles \cite{sc1}, \cite{ba}.
\item There is only a finite number of maximal surfaces of genus $g$ \cite{sc1}, \cite{ba}.
\item The systole function $\sy(\cdot)$ is a topological Morse function on the moduli space \cite{ak}.
\item All systoles of maximal surfaces are non-separating (see \cite{pa1}, \textbf{Claim} on p. 336).
\end{itemize}

An open question is, whether $\msy(\cdot)$ is a monotonously increasing function with respect to the genus. Though we can not prove or disprove this result, we can at least show the following intersystolic inequalities:
\begin{thm} Let $\msy(g)$ be the maximal value that $\sy(\cdot)$ can attain among the compact hyperbolic Riemann surfaces of genus $g \geq 2$.
\begin{enumerate}
\item $\msy(k(g-1)+1) > \msy(g)$ for $k \in \N \backslash \{0,1\}$.
\item $\msy(g+1) > \frac{\msy(g)}{2}.$
\item If $\msy(g_2) \geq \msy(g_1)$, then $\msy(g_1 +g_2 -1) > \min \{ \frac{\msy(g_2)}{2}, \msy(g_1)\}$.
\end{enumerate}
\label{thm:intermsys}
\end{thm}
We note that \textbf{Theorem \ref{thm:intermsys}-1} without a sharp inequality can be obtained by constructing a normal cover of a maximal surface in a similar fashion (see \textbf{Note \ref{thm:kcover_note}}). This is due to the well known fact that the injectivity radius does not decrease in a normal cover of a closed Riemannian surface.\\
The theorem is obtained by cutting and pasting maximal surfaces to construct compact hyperbolic Riemann surfaces with large systoles. Here the main tool is \textbf{Lemma \ref{thm:sys_coll}}, a collar lemma for systoles. As a result we obtain from \textbf{Theorem \ref{thm:intermsys}-1} or the covering construction:
If $S$ is a compact Riemann surface of genus $g$, such that $\sy(S) \ge \frac{4}{3} \log(g)-c_0$, then for $l=k\cdot (g-1)+1$
\[
      \msy(l) \ge \frac{4}{3} \log(l)- \left(\frac{4}{3} \log(k) +c_0 \right) =  \frac{4}{3} \log(l)- c(k) \text{ \ \ for all \ \ }    k \ll g.
\]
The concrete examples of genera $g$ for which $\msy(g) \ge \frac{4}{3} \log(g)-c_0$ is shown are sparse. The above inequality, however, suggests that at least a slightly lower bound holds for a large number of genera $g$.\\ 
Furthermore, by construction, we obtain a continuous parameter family $\left(S_t \right)_{t \in (-\frac{1}{2},\frac{1}{2}]^k}$ of compact hyperbolic Riemann surfaces of genus $k(g-1)+1$, such that
\[
      \sy(S_t) = \msy(g).
\]
This shows that though the systoles of these surfaces are large, none of these surfaces can be maximal, as this would be a contradiction to the finiteness of the number of these surfaces.\\
In the following table Tab.~\ref{tab:RS_large_sys} we give a summary of compact hyperbolic Riemann surfaces of genus $g \leq 25$ with maximal known systoles thus providing a reference and benchmark for further studies. Most of these are constructed using the examples presented in \cite{ca},\cite{ksv1}, \cite{sc1} and \cite{sc4} by applying \textbf{Theorem \ref{thm:intermsys}} or the covering argument. 

\begin{table}[htbp]
\begin{center}
\begin{tabular}{|c|c|c|c|c|c|}
\hline
\multicolumn{1}{|c|}{genus $g$} & \multicolumn{1}{|c|}{ \parbox[t]{4cm}{surface
(name and/or \\ constructed from)}} & \multicolumn{1}{|c|}{\parbox[t]{3cm}{(systole length)/ \\  $\log(g)$}} & \multicolumn{1}{|c|}{\parbox[t]{1cm}{systole \\ length}} & \multicolumn{1}{|c|}{$2\log(4g-2)$} &\multicolumn{1}{|c|} {reference} \\ \hline
2 & $M^*$ & 4.41 & 3.06 & 3.58 & \cite{sc1}(Bolza) \\
3 & $M(3)$ & 3.63 & 3.98 & 4.61 &\cite{sc1} (Wiman) \\
4 & $M(4)$ & 3.34 & 4.62 & 5.28 & \cite{sc1} \\
5 & $S_5$ & 3.05 & 4.91 & 5.78 & \cite{sc4} \\
6 & $I_6$ & 2.85 & 5.11 & 6.18 & \cite{ca} \\
7 & $H_7$ (via PSL(2,8)) & 2.98 & 5.80 & 6.52 & \cite{ksv1} \\
8 & $7 \times M^*$ & 1.47 & 3.06 & 6.80 & normal cover \\
9 & $2 \times S_5$ & 2.24 & 4.91 & 7.05 & normal cover \\
10 & $3 \times M(4)$ & 2.01 & 4.62 & 7.28 & normal cover \\
11 & $I(x|z)$ & 2.49 & 5.98 & 7.48 & \cite{sc1}\\
12 & $11 \times M^*$ & 1.23 & 3.06 & 7.66 & normal cover \\
13 & $M_{13}$ (via  $2 \times H_7$) & 2.26 & 5.80 & 7.82 & normal cover \\
14 & $H_{14}$ (via $PSL(2,13)$) & 2.61 & 6.89 & 7.98 & \cite{ksv1} \\
15 & $7 \times M(3)$ & 1.47 & 3.98 & 8.12 & normal cover \\
16 & $3 \times I_6$ & 1.84 & 5.11 & 8.25 & normal cover \\
17 & $H_{17}$ (via $(C_2)^3.PSL(2,7)$) & 2.69 & 7.61 & 8.38 & \cite{ksv1} \\
18 & via $H_{17}$ & 1.32 & 3.80 & 8.50 & Th. 1.1-2 \\
19 & $M_{19}$ (via $3 \times H_7$) & 1.97 & 5.80 & 8.61 & normal cover \\
20 & $H_{17},M(4)$ & 1.27 & 3.80 & 8.71 & Th. 1.1-3 \\
21 & $2 \times I(x|z)$ & 1.96 & 5.98 & 8.81 & normal cover \\
22 & $7 \times M(4)$ & 1.50 & 4.62 & 8.91 & normal cover \\
23 & $B_1$ & 2.04 & 6.39 & 9.00 & \cite{sc1} \\
24 & via $B_1$ & 1.01 & 3.19 & 9.09 & Th. 1.1-2 \\
25 & $2 \times M_{13}$ & 1.80 & 5.80 & 9.17 & normal cover \\ \hline
\end{tabular}
\end{center}
\caption{Compact hyperbolic Riemann surfaces with large systoles of genus $g \leq 25$.}
\label{tab:RS_large_sys}
\end{table}

\newpage
It follows furthermore from the known examples in genus $2$ and $3$ and \textbf{Theorem \ref{thm:intermsys}-1} or \textbf{Note \ref{thm:kcover_note}}: Let $\msy(g)$ be the maximal value that $\sy(\cdot)$ can attain among the compact hyperbolic Riemann surfaces of genus $g \geq 2$. Then
\begin{equation}
 3.06\simeq \msy(2)   <  \msy(g)  \text{ \ \ for all  } g\text{ \ and \ \ \ }  3.98 \leq \msy(3)   <  \msy(g) \text{ \ \ if }  g\text{\ odd. \ }
\label{eq:msys_general} 
\end{equation}
In Section 3 we use the same methods to derive similar inequalities for non-compact hyperbolic surfaces with \textit{cusps} (see \cite{bu2}, \textbf{Example 1.6.8} for an exact definition). Let $S$ be a hyperbolic Riemann surface of signature $(g,n)$, i.e. of genus $g$ and with $n$ cusps. Let $\msy(g,n)$ be the value
\[
     \msy(g,n) = \sup \{ \sy(S) \mid S \text{ non-compact hyperbolic Riemann surface of signature } (g,n)\}.
\]
As in the case of the compact hyperbolic surfaces this supremum is a maximum due to Mahler's compactness theorem. The known bounds for the value of $\msy(g,n)$ are the following: For $(g,n) \neq (0,3)$ and $n \geq 2$ 
\begin{equation}
      2 \arcsinh(1) \leq \msy(g,n) \leq 4 \log \left(\frac{12g-12+6n}{n}\right).
\label{eq:log_bound2}
\end{equation}
The upper bound follows again from an area argument (see \cite{sc2}), whereas the lower bound is due to the collar lemma for Riemann surfaces (see \cite{bu2}, \textbf{Theorem 4.1.1}). More refined estimates have been obtained recently in \cite{fp}.\\
Heuristically speaking, if we fix the genus $g$ and increase the number $n$ of cusps continuously, then some of the cusps must move closer together and there will be a simple closed geodesic of bounded length surrounding two or several cusps. Such a geodesic is separating the surface into two parts. Indeed it can be shown that:
\begin{itemize}
\item Let $S_{max}$ be a maximal hyperbolic Riemann surface of signature $(g,n)$, where $n \geq 25g$. Then  all systoles of $S_{max}$ are separating. (\cite{pa1}, \textbf{Proposition 3.1}).
\end{itemize}
This will be important in the statement of the following theorem.
\begin{thm} Let $S_{max}$ be a maximal non-compact hyperbolic Riemann surface of signature $(g,n) \neq (0,4)$, where $3g-3 + n > 0$ and $n \geq 2$ that has a separating systole. Let $\sy(S_{max})=\msy(g,n)$ be its systole length. Then $\msy(g,n) \geq 4\arcsinh(1)$. Furthermore
\begin{enumerate}
\item $\msy(2g, 2n-4) > \min\left\{\msy(g,n),\max\{5.276,\msy(g,n)-4\arcsinh \left(\frac{1}{\sinh(\frac{\msy(g,n)}{4})}\right)\}\right\}$.
\item $\msy(2g+1,n-4) > \frac{\msy(g,n)}{2}.$
\item $\msy(g+1,n-2) > \frac{\msy(g,n)}{3}.$
\end{enumerate}
\label{thm:intermsys2}
\end{thm}
Here the lower bound $\msy(g,n) \geq 4\arcsinh(1)$ is due to $\cite{gs}$. We furthermore construct normal covers of maximal non-compact hyperbolic surfaces of signature $(g,n)$ and show that the systole length does not decrease in these surfaces. This enables us to prove: 
\begin{thm} Let $\msy(g,n)$ be the maximal value that $\sy(\cdot)$ can attain among the non-compact hyperbolic Riemann surfaces of signature $(g,n)$, where $3g-3 + n > 0$ and $g \geq 1$. Then
$$\msy(k(g-1)+1, k n) \geq \msy(g,n) \text{ \ \ for \ \  } k \in \N \backslash \{0,1\}$$.
\label{thm:intermsys3}
\end{thm}

\section*{Acknowledgment}
The second author has been supported by the Alexander von Humboldt foundation. We would like to thank Guillaume Bulteau for helpful discussions. We would like to thank the anonymous referee from JGEOM for his constructive comments which substantially helped to improve the quality of the article.

\section{Construction of compact hyperbolic Riemann surfaces with large systoles from maximal surfaces}

By abuse of notation we will denote the length of a geodesic arc by the same letter as the arc itself in the following sections. In this section we treat the case of compact hyperbolic surfaces and a surface of signature $(g,n)$ is a compact surface of genus $g$ with $n$ disjoint boundary components, each of which is a smooth simple closed geodesic. We will first prove a collar lemma for systoles. Then we will present the construction of a certain normal cover of Riemann surfaces. In this cover the injectivity radius does not decrease. Finally this construction will be used to show \textbf{Theorem \ref{thm:intermsys}}.\\ 
\\
Let $\eta$ be a simple closed geodesic on $S$. Let $\omega_{\eta}$ be the supremum of all $w$, such that the geodesic arcs of length $w$ emanating  perpendicularly from  $\eta$ are pairwise disjoint. We define a \textit{collar} $C_w(\eta)$ around $\eta$ of width $w<\omega_{\eta}$ by
\[
C_w(\eta)=\left\{p \in M \mid \dist(p,\eta) < w \right\}.
\]
We call a collar $C_w(\eta)$ of width  $w=\omega_{\eta}$ the \textit{maximal collar} of $\eta$. We call a \textit{half-collar} $H_w(\eta)$ of width $w$ one of the two parts of a collar $C_w(\eta)$ that we obtain by cutting $C_w(\eta)$ along $\eta$.\\
We first show the following collar lemma for systoles of compact surfaces.
\begin{lem} Let $\alpha$ be a systole of a compact hyperbolic Riemann surface $S$ of genus $g \geq 2$. Then the maximal collar $C_{\omega_{\alpha}}(\alpha)$ of $\alpha$ has width
     $$ \omega_{\alpha} > \frac{\alpha}{4}.$$
\label{thm:sys_coll}
\end{lem}
\begin{proof} Let $\alpha$ be a systole of a compact hyperbolic Riemann surface $S$ of genus $g \geq 2$. The closure $\overline{C_{\omega_{\alpha}}(\alpha)}$ of the maximal collar of $\alpha$  self-intersects in a point $p$. There exist two geodesic arcs $\delta'$ and $\delta''$ of length $\omega_{\alpha}$ emanating from $\alpha$ and perpendicular to $\alpha$ having the endpoint $p$ in common. These two arcs form a smooth geodesic arc $\delta$. The endpoints of $\delta$ on $\alpha$ divide $\alpha$ into two parts. We denote these two arcs by $\alpha'$ and $\alpha''$. Let without loss of generality $\alpha'$ be the shorter arc of these two. We have that
\[
   \alpha' \leq \frac{\alpha}{2}.
\]
Let $\beta$ be the simple closed geodesic in the free homotopy class of $\alpha'\cdot \delta$. We have that
\[
   \beta < \alpha' + \delta \leq \frac{\alpha}{2} + 2 \omega_{\alpha}.
\]
Now if $\omega_{\alpha} \leq  \frac{\alpha}{4}$ then it follows from this inequality that $ \beta < \alpha$. A contradiction to the minimality of $\alpha$.
\end{proof}

Hence each systole $\alpha$ in a compact hyperbolic Riemann surface $S$ has a collar $C_{\frac{\alpha}{4}}(\alpha)$ of width $\frac{\alpha}{4}$ which is embedded in $S$. From this fact we also obtain an upper bound for the length of a systole via an area argument. However, this estimate is not better than the one given in the introduction (see inequality (\ref{eq:log_bound})).\\
The following construction of a normal cover can be used to provide a version of \textbf{Theorem \ref{thm:intermsys}-1}, where the inequality is not sharp. Even if the proof of this note is well known, we include it here for better comprehension of the following constructions. 
\begin{nte} Let $S$ be a compact hyperbolic Riemann surface of genus $g \geq 2$. There is a normal cover $\tilde{S}$ of order $k$ and of genus $k(g-1)+1$, such that 
$$\sy(\tilde{S}) \geq \sy(S).$$
\label{thm:kcover_note}
\end{nte}
\begin{figure}[h!]
\SetLabels
\L(.30*.86) $S_1^c$\\
\L(.68*.86) $S_2^c$\\
\L(.68*.11) $S_3^c$\\
\L(.30*.11) $S_4^c$\\
\L(.25*.49) $\alpha^{1}_2 \sim \alpha^4_1$\\
\L(.46*.84) $\alpha^{1}_1 \sim \alpha^2_2$\\
\L(.67*.49) $\alpha^{2}_1 \sim \alpha^3_2$\\
\L(.46*.12) $\alpha^{4}_2 \sim \alpha^3_1$\\
\endSetLabels
%\ShowGrid
\AffixLabels{%
\centerline{%
\includegraphics[height=8cm,width=8cm]{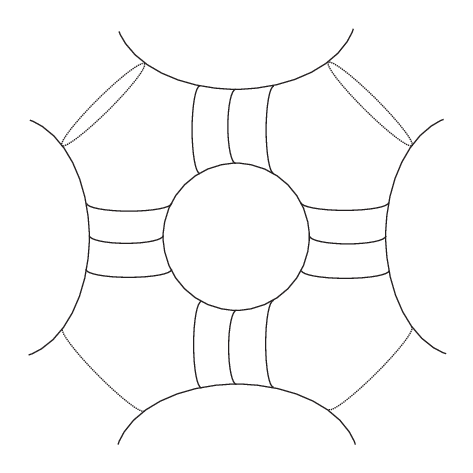}}}
\caption{Four copies $\left(S_i^c\right)_{i=1,..,4}$ of $S^c$ with identified boundaries.}
\label{fig:paste}
\end{figure}
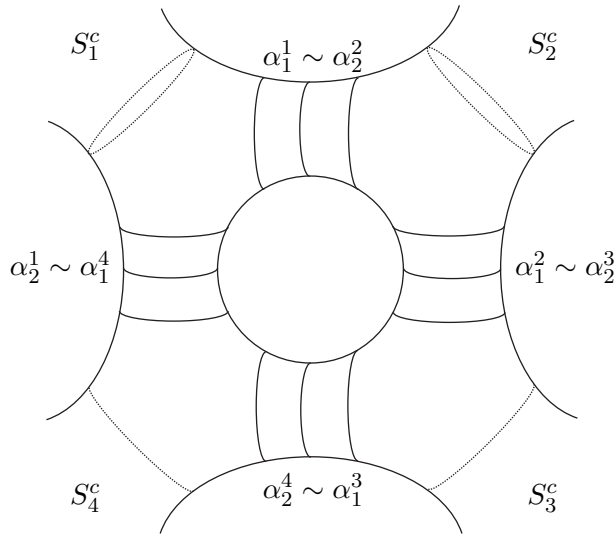
\begin{proof}
Let $S$ be a  compact hyperbolic Riemann surface of genus $g \geq 2$ and let $\alpha$ be a non-separating simple closed geodesic of $S$. Let $\mathcal{P}$ be a partition of $S$ into surfaces of signature $(0,3)$, or \textit{Y-pieces}, such that $\alpha$ is the boundary curve of a Y-piece. Let $FN_S$ be a set of \textit{Fenchel-Nielsen coordinates} corresponding to $\mathcal{P}$ (see \cite{bu2}, p. 27-30) and let $tw_\alpha$ be the \textit{twist parameter} at $\alpha$.\\ 
To prove our note we construct a new surface $\tilde{S}$ of genus $k(g-1)+1$ from $S$ such that
\[
\sy(\tilde{S}) \geq \sy(S).
\]
To this end we cut open $S$ along $\alpha$ and call the surface obtained this way $S^c$. The signature of $S^c$ is $(g-1,2)$. Let $\alpha^c_1$ and $\alpha^c_2$ be the two boundary geodesics of $S^c$. We recall that  by gluing $S^c$ with twist parameter $tw_\alpha$ along $\alpha^c_1$ and $\alpha^c_2$ we obtain $S$.\\ 
Take $k$ copies $\left(S_i^c\right)_{i=1,..,k}$ of the surface $S^c$ and let $\alpha^i_1$ and $\alpha^i_2$ be the boundary geodesics of $S_i^c$, such that $\alpha^i_1$ is the copy of $\alpha^c_1$ and $\alpha^i_2$ is the copy of $\alpha^c_2$. We identify the boundaries of the different $\left(S_i^c\right)_{i=1,..,k}$ in the following way
\begin{equation}
    \alpha^{k}_1 \sim \alpha^1_2  { \ \ and \ \ } \alpha^{i}_1 \sim \alpha^{i+1}_2  \text{ \ for \ } i=1,...,k-1
\label{eq:paste1}
\end{equation}
choosing the $k$ twist parameters $(t_j)_{j=1,..,k}$, such that
\[
          t_j = tw_\alpha \text{ \ \ for all \ \ } k \in \{1,..,k\}. 
\]
We denote the surface of genus $k(g-1)+1$ obtained according to this pasting scheme as
\[
    \tilde{S}= S_1^c + S_2^c + ... + S_k^c~\mod (\ref{eq:paste1}) \text{ \ (see Fig.~\ref{fig:paste})}.
\]
We denote by $\alpha_i$ the image of $\alpha^{i}_1 \subset S_i^c$ in $\tilde{S}$. For $i \in \{1,..,k\}$, let
\[
\phi_i: S_i^c \backslash \partial S_i^c  \rightarrow S_{i+1 \mod k}^c \backslash \partial S_{i+1 \mod k}^c
\]
be the natural isometry on the interior of the $S_i^c$. As all twist parameters in the pasting are chosen to be $tw_\alpha$ the $(\phi_i)_{i=1,..,k}$ can be extended to an isometry $\tilde{\phi}:\tilde{S} \rightarrow \tilde{S}$ and we have that 
\[
    \tilde{S}~\mod\tilde{\phi}  \simeq S.
\] 
Then $\langle \tilde{\phi} \rangle$ is the cyclic group of order $k$ generated by $\tilde{\phi}$ and $\tilde{S}$ is a normal cover of $S$.\\
Now it is well known that the injectivity radius $r_{inj}(M)$ of a closed Riemannian surface $M$ does not decrease in a normal cover. But in the case of a compact hyperbolic Riemann surface $S'$ we have by \cite{bu2}, \textbf{Theorem 4.1.5} that $r_{inj}(S') = \frac{\sy(S')}{2}$. Hence 
\[
    \sy(\tilde{S}) = 2r_{inj}(\tilde{S}) \geq 2r_{inj}(S) = \sy(S).
\]
This concludes the proof of our note.     
\end{proof}

\textbf{proof of Theorem \ref{thm:intermsys}} \\
\\
\textit{1. $\msy(k(g-1)+1) > \msy(g)$.}\\
\\
Let $S_{max}$ be a globally maximal compact hyperbolic Riemann surface of genus $g \geq 2$. To prove the first part of the theorem we construct a new surface $S'$ of genus $k(g-1)+1$ from $S_{max}$ such that
\[
\msy(g) = \sy(S_{max}) = \sy(S').
\]
Let $\alpha$ be a systole of $S_{max}$. As all systoles of maximal surfaces are non-separating (see \cite{pa1}, \textbf{Claim} on p. 336), $\alpha$ is a non-separating simple closed geodesic. We proceed as in the proof of \textbf{Note \ref{thm:kcover_note}}, setting $S_{max}=S$ and  $S'=\tilde{S}$. However, this time we choose the $k$ twist parameters 
\[
 (t_j)_{j=1,..,k}, \text{ \ \ \ } t_j  \in (-\frac{1}{2},\frac{1}{2}]
\]
in the pasting (\ref{eq:paste1}) freely  to obtain the closed surface $S'$ of genus $k(g-1)+1$. Recall that $\alpha_i$ is the image of $\alpha^{i}_1 \subset S_i^c$ in $S'$. Due to our construction and \textbf{Lemma \ref{thm:sys_coll}} each $\alpha_i$ has an embedded collar $C_i$ of width $\frac{\alpha}{4}$ (see Fig.~\ref{fig:paste}).\\
We now show that any simple closed geodesic $\eta$ in $S'$ has length bigger than or equal to $\alpha$.
Therefore we distinguish two cases: either $\eta$ intersects at least one of the $(\alpha_i)_{i=1,..,k}$ transversally or not.\\
Consider the first case. If $\eta$ intersects one of the $(\alpha_i)_{i=1,..,k}$, say $\alpha_j$, then, due to the geometry of hyperbolic cylinders, $\eta$ traverses $C_j$. Now, to be a closed curve $\eta$ has to traverse at least two times the same cylinder $C_j$ or has to traverse $C_j$ and a different cylinder $C_l$. Therefore its length is bigger than $2 \cdot 2\omega_{\alpha} = 2 \cdot \frac{\alpha}{2} = \alpha$. In this case we have that
\[
    \eta \geq \alpha 
\]
which implies that $\sy(S')= \alpha$.\\
In the second case, a simple closed geodesic that intersects none of the $(\alpha_i)_{i=1,..,k}$ transversally is either contained in the interior of one of the $\left(S_i^c\right)_{i=1,..,k}$ or is one of the $(\alpha_i)_{i=1,..,k}$. In any case we have that
\[
    \eta \geq \alpha 
\]
thus $\sy(S') = \alpha$.\\
Hence in any case we have that $\msy(g) = \sy(S_{max}) =  \alpha = \sy(S')$.\\
 Setting the twist parameters $t=(t_1,..,t_k)$ arbitrarily, we see that $S'$ belongs to a continuous family $\left(S_t \right)_{t \in (-\frac{1}{2},\frac{1}{2}]^k}$ of compact hyperbolic Riemann surfaces of genus $k(g-1)+1$, such that
\[
      \sy(S_t) = \msy(g).
\]
This shows that, though the systoles of the surfaces in this family are large, these can not be maximal, as this would be a contradiction to the finiteness of the number of these surfaces. Hence $\msy(k(g-1)+1)>\sy(S_t)=\msy(g)$.\\
\\
\textit{2. $\msy(g+1) > \frac{\msy(g)}{2}$.}\\
\\
We call a surface of signature $(1,2)$ a \textit{F-piece}. Let $b_1$ and $b_2$ be the boundary geodesics of a F-piece. Let $F_{max}$ be a surface whose interior systole has maximal length among all F-pieces with boundaries $b_1$ and $b_2$ of equal length
$$b_1=b_2=b.$$
Denote by $s=\sy(F_{max})$ the length of the interior systole of $F_{max}$. It was shown in \cite{sc1}, \textbf{Theorem 5.1.} that
\begin{equation}
 2\cosh(\frac{s}{2})^3 - 3\cosh(\frac{s}{2})^2 - (\cosh(\frac{b}{2}) + 1)\cosh(\frac{s}{2}) - \cosh(\frac{b}{2})=0.    
\label{eq:sys_F}
\end{equation}
This implies that $\sy(F_{max}) = s > \frac{b}{2}$.\\
To prove the second part of the theorem we construct a new surface $S''$ of genus $g+1$ from a globally maximal surface $S_{max}$ of genus $g \geq 2$ and $F_{max}$ such that
\[
 \sy(S'') \geq \frac{\sy(S_{max})}{2}.
\]
To construct $S''$, we take the surface $S^c$ of signature $(g-1,2)$ obtained by cutting open $S_{max}$ along a systole $\alpha$ and paste the surface $F'_{max}$ with boundary geodesics of length $\alpha$ along the two boundary geodesics of $S^c$.\\
We denote by $\alpha_1$ and $\alpha_2$ the image of the two boundary geodesics of the embedded F-piece in $S''$ which is isometric to $F'_{max}$. Due to our construction each $\alpha_i$ has an embedded half-collar $H_i \subset S^c$ of width $\frac{\alpha}{4}$.\\
We now show that any simple closed geodesic $\eta$ in $S''$ has length bigger than  $\frac{\alpha}{2}$.\\
Therefore we distinguish two cases: either $\eta$ intersects either $\alpha_1$ or $\alpha_2$ transversally or not. Consider the first case. If $\eta$ intersects one of the $(\alpha_i)_{i=1,2}$, say $\alpha_1$, transversally, then, due to the geometry of the half-collar, $\eta$ traverses $H_1$. Furthermore $\eta$ has to traverse at least two times the same half-collar $H_1$ or has to traverse $H_1$ and $H_2$ to be a closed curve. In this case its length is bigger than or equal to $2 \cdot \omega_{\alpha} = 2 \cdot \frac{\alpha}{4} = \frac{\alpha}{2}$. Hence in this case we obtain
\[
    \eta \geq \frac{\alpha}{2}
\]
and therefore $\sy(S'') \geq \frac{\alpha}{2}$. (In fact this inequality is strict, because the part of $\eta$ in the F-piece must have strictly positive length.)\\
We now consider the second case. Any simple closed geodesic that intersects neither $\alpha_1$ nor $\alpha_2$ transversally is either contained in the interior of $S^c$, the interior of $F'_{max}$ or is one of the $(\alpha_i)_{i=1,2}$. Here we obtain the lower bound for $\eta$ from the lower bound on $\sy(F'_{max})$. It follows from Equation (\ref{eq:sys_F}) that
\[
    \eta > \frac{\alpha}{2}  
\]
thus $\sy(S') \geq \frac{\alpha}{2}$.\\
In any case we obtain that $\msy(g+1) > \sy(S'') \geq \frac{\sy(S_{max})}{2} = \frac{\msy(g)}{2}$.\\
As in the previous part, the inequality in \textbf{Theorem \ref{thm:intermsys}-2} is strict due to the fact that the construction does not depend on the twist parameters.\\
\\
\textit{3. If $\msy(g_2) \geq \msy(g_1)$, then $\msy(g_1 +g_2 -1) > \min \{ \frac{\msy(g_2)}{2}, \msy(g_1)\}$.}\\
\\
To prove the final statement we take two maximal surfaces $S^1_{max}$ and $S^2_{max}$ of genus $g_1$ and $g_2$, respectively. For $i \in \{1,2\}$ we cut $S^i_{max}$ open along a systole $\alpha^i$ and call the surface obtained in this way $S^{ci}$. As $\msy(g_2) \geq  \msy(g_1)$ we have that $\alpha^2 \geq \alpha^1$.\\

Now, we can not directly paste these surfaces together, as the boundary length is different. However, by \cite{pa2}, \textbf{Theorem 1.1.}, we can construct a comparison surface $S^{p1}$ for $S^{c1}$ of signature $(g_1-1,2)$, such that
\begin{itemize}
\item All interior geodesics of $S^{p1}$ are longer than $\alpha^1$.
\item The boundary geodesics $\gamma_1$ and $\gamma_2$  of $S^{p1}$ have length $\alpha^2$.
\end{itemize}
We identify the open boundaries of $S^{p1}$ and $S^{c2}$ to obtain the surface $S^{12}$ of genus $g_1 + g_2 -1$. Furthermore the two boundary geodesics $\alpha^2_1$ and $\alpha^2_2$ of $S^{c2}$ have both an embedded half-collar of width $\frac{\alpha^2}{4}$ in $S^{c2}$. Due to the properties of $S^{12}$ we can apply similar arguments as in the previous case of the surface $S''$ to show that
\[
   \msy(g_1+g_2 -1) >  \sy(S^{12}) \geq \min \{ \frac{\alpha^2}{2}, \alpha^1 \} = \min \{ \frac{\msy(g_2)}{2}, \msy(g_1) \}.
\]
This concludes the proof of \textbf{Theorem \ref{thm:intermsys}}. \hfill  $\square$ 

\section{Construction of hyperbolic Riemann surfaces with cusps with large systoles}

In this section we denote by a hyperbolic (Riemann) surface $S$ of signature $(g,n)$ a non-compact hyperbolic Riemann surface of genus $g$ with $n$ cusps and assume that $3g-3+n > 0$. A surface with a cusp can also be interpreted as a surface with a degenerated boundary of length zero. In view of this we will denote in this section by a hyperbolic surface $S$ of signature $(g,n,k)$ a surface of genus $g$ with $n$ disjoint boundary components, of which $k$ are cusp points at infinity and of which $n-k$ are smooth simple closed geodesics.\\
\\           
To prove \textbf{Theorem \ref{thm:intermsys2}}, we will first show a collar lemma for separating systoles of hyperbolic surfaces with cusps. More specifically, we show that a systole with a large collar exists if a separating systole is intersected by another systole. Contrary to the case of the compact hyperbolic surfaces, this collar lemma applies only in this special situation. With the help of this lemma we then prove \textbf{Theorem \ref{thm:intermsys2}}. Then we explain shortly why no such strong collar lemma should hold if the systole of hyperbolic surface with cusps is non-separating. Finally we prove \textbf{Theorem \ref{thm:intermsys3}} using the construction of a normal cover from Section 2.  
\begin{lem} Let $S$ be a non-compact hyperbolic Riemann surface of signature $(g,n)\neq (0,4)$, where $3g-3+n > 0$ and $n \geq 2$. If $S$ has a separating systole that is intersected by another systole, then there exists a separating systole $\alpha$ that bounds a surface $Y'$ of signature $(0,3,2)$. The maximal half-collar $H_{\omega_{\alpha}}(\alpha) \subset S$ of $\alpha$ on the side opposite of $Y'$ has width
\[
      \omega_{\alpha} > \min\left\{\frac{\alpha}{4},\max\{1.319, \frac{\alpha}{4} - \arcsinh\left(\frac{1}{\sinh(\frac{\alpha}{4})}\right)\} \right\}.
\]
\label{thm:sys_coll_cusp}
\end{lem}
\begin{proof}
Let $S$ be a hyperbolic surface of signature $(g,n)$, where $(g,n)$ satisfies the conditions of the lemma. We first show that the surface $S$ has a separating systole that bounds a Y-piece $Y'$ of signature $(0,3,2)$. Therefore we will make use of the following facts:
\begin{itemize}
\item If $S'$ is a hyperbolic surface of signature $(g,n)$ then two systoles intersect at most twice (see \cite{fp}, \textbf{Proposition 3.2}).
\item If $S'$ is a hyperbolic surface of signature $(g,n)$ and two systoles intersect twice, then one of them bounds two cusps (see \cite{fp}, \textbf{Proposition 3.3}).
\item If two systoles in a hyperbolic surface $S'$ of signature $(g,n)$ intersect twice, then 
            $$\sy(S') > 4\arcsinh(1) > 3.525 \text{ \ \ (see \cite{gs}).}$$ 
\end{itemize}
Let $\alpha$ be a separating systole in $S$ that is intersected by another systole $\beta$.
Now suppose $\alpha$ and $\beta$ intersect once. Then by surface topology both are homologically non-trivial and both must be non-separating, which is a contradiction to the fact that $\alpha$ is separating. Hence $\alpha$ and $\beta$ intersect twice. It follows from \cite{fp}, \textbf{Proposition 3.3} that either $\alpha$ or $\beta$ bounds a surface $Y'$ of signature $(0,3,2)$. Let without loss of generality $\alpha$ be the boundary geodesic of $Y'$. We now prove the existence of the half-collar.\\
To this end we cut open $S$ along the systole $\alpha$. As $\alpha$ is separating, $S$ decomposes into two parts. Let $R$ be the part which is a surface of signature $(g,n-1,n-2)$. Consider the closure $\overline{H_{\omega_{\alpha}}(\alpha)}$ of the half-collar $H_{\omega_{\alpha}}(\alpha)$ in $R$. This maximal half-collar of $\alpha$  self-intersects in a point $p$. There exist two geodesic arcs $\delta'_1$ and $\delta''_1$ of length $\omega_{\alpha}$ emanating from $\alpha$ and perpendicular to $\alpha$ having the endpoint $p$ in common. These two arcs form a smooth geodesic arc $\delta_1$ (see Fig.~\ref{fig:X12}). The endpoints of $\delta_1$ on $\alpha$ divide $\alpha$ into two parts. We denote these two arcs by $\alpha'$ and $\alpha''$. Let without loss of generality $\alpha'$ be the shorter arc of these two. Then
\[
   \alpha' \leq \frac{\alpha}{2}.
\]
Let $\gamma$ be the simple closed geodesic in the free homotopy class of $\alpha'\cdot \delta_1$.
We have that
\begin{equation}
   \gamma < \alpha' + \delta_1 \leq \frac{\alpha}{2} + 2 \omega_{\alpha}.
\label{eq:omega_cusp_sep}
\end{equation}
Now two cases can occur: Either $\gamma$ is a simple closed geodesic of non-zero length or $\gamma$ has length zero i.e. defines a cusp.\\ 
\\
\textit{Case 1: $\gamma$ has non-zero length}\\ 
\\
In this case, we conclude that 
\[
     \omega_{\alpha} > \frac{\alpha}{4}.
\]  
Otherwise it would follow from Equation (\ref{eq:omega_cusp_sep}) that $\gamma$ is a simple closed geodesic of length smaller than the systole. A contradiction. This settles our claim in the first case.\\
\\
\textit{Case 2: $\gamma$ has length zero}\\ 
\\
It remains the second case, where $\gamma$ defines a cusp. In this case we have to work a bit harder. Let $\eta$ be the simple closed geodesic in the free homotopy class of $\alpha''\cdot \delta_1$. Now $\eta$ can not have length zero, because then $S$ would be a surface of signature $(0,4,4)$, which we excluded. Therefore let $Y$ be the surface of signature $(0,3,1)$, whose boundary curves are the cusp defined by $\gamma$, $\eta$ and $\alpha$. Then $ Y \cup Y' = X^1 \subset S$ defines a surface $X^1$ of signature $(0,4,3)$ (see Fig \ref{fig:X12}).
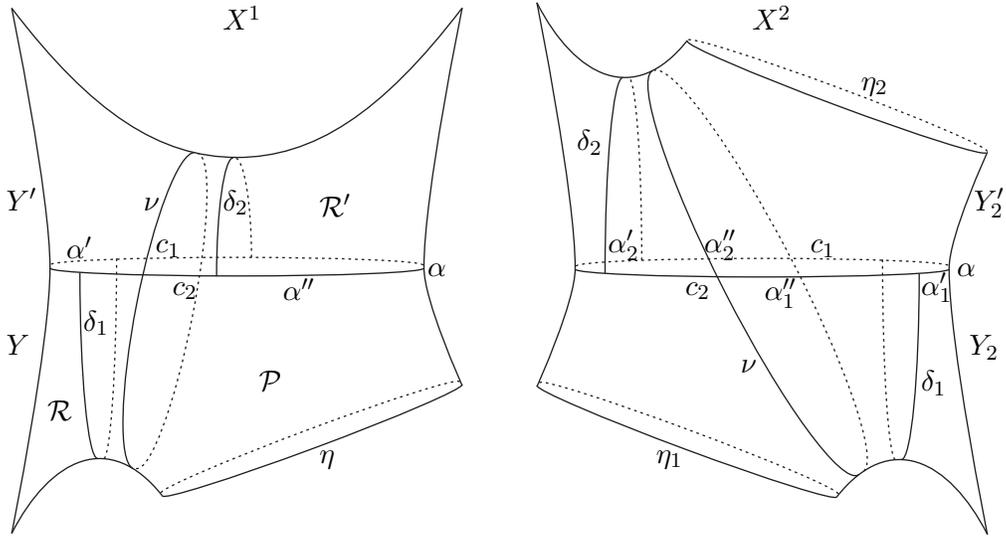
\begin{figure}[h!]
\SetLabels
\L(.27*.90) $X^1$\\
\L(.09*.60) $Y'$\\
\L(.09*.36) $Y$\\
\L(.44*.49) $\alpha$\\
\L(.14*.52) $\alpha'$\\
\L(.32*.45) $\alpha''$\\
\L(.35*.18) $\eta$\\
\L(.27*.60) $\delta_2$\\
\L(.35*.59) $\mathcal{R'}$\\
\L(.15*.40) $\,\delta_1$\\
\L(.12*.25) $\,\mathcal{R}$\\
\L(.30*.30) $\mathcal{P}$\\
\L(.22*.46) $\,\,c_2$\\
\L(.21*.53) $\,c_1$\\
\L(.20*.60) $\,\nu$\\
%%%%%%%%%%%%%%%%%%%%%%%%%%%%%%%%%%%
\L(.71*.90) $X^2$\\
\L(.89*.60) $\,Y'_2$\\
\L(.89*.36) $Y_2$\\
\L(.88*.49) $\alpha$\\
\L(.63*.18) $\eta_1$\\
\L(.85*.30) $\delta_1$\\
\L(.65*.46) $\,c_2$\\
\L(.85*.46) $\alpha'_1$\\
\L(.72*.45) $\alpha''_1$\\
%%%%%%%%%%%%%%%%%%%%%%%%%%%%%%%%%%%
\L(.80*.80) $\eta_2$\\
\L(.56*.70) $\,\delta_2$\\
\L(.59*.53) $\alpha'_2$\\
\L(.67*.53) $\alpha''_2$\\
\L(.75*.53) $\,\,c_1$\\
\L(.70*.33) $\nu$\\
\endSetLabels
%\ShowGrid
\AffixLabels{%
\centerline{%
\includegraphics[height=8cm,width=14cm]{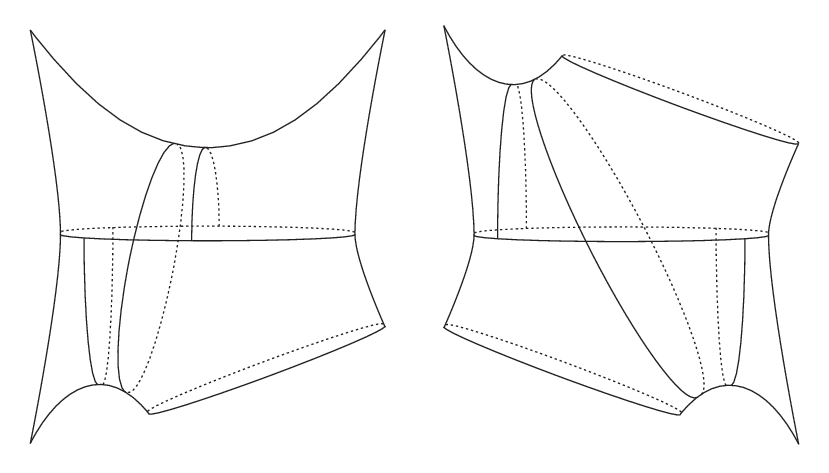}}}
\caption{The surface $X^1$ of signature $(0,4,3)$ and the surface $X^2$ of signature $(0,4,2)$.}
\label{fig:X12}
\end{figure}
Let $\delta_2$ be the shortest geodesic arc in $Y'$ connecting $\alpha$ with itself. By symmetry the two endpoints of $\delta_2$ on $\alpha$ divide $\alpha$ into two arcs, both of length $\frac{\alpha}{2}$. Let $\mathcal{R}'$ be a geodesic hyperbolic quadrilateral in $Y'$ with three angles of value $\frac{\pi}{2}$ that has one side on $\delta_2$ of length $\frac{\delta_2}{2}$, one side on $\alpha$ of length $\frac{\alpha}{4}$ and two sides of infinite length meeting in the point at infinity defined by one of the cusps. $\mathcal{R}'$ can be interpreted as a degenerated right-angled hyperbolic pentagon. It follows from the geometry of right-angled pentagons (see \cite{bu2}, p. 454) that 
\begin{equation}
\sinh(\frac{\delta_2}{2}) \cdot \sinh(\frac{\alpha}{4}) =1.
\label{eq:delta2_collar}
\end{equation}
We have furthermore that 
\[
\delta_1 \geq \alpha'  
\]
as otherwise $ \alpha \leq \eta < \delta_1 + \alpha'' < \alpha' + \alpha'' = \alpha$.\\
Let $\mathcal{R}$ be a geodesic hyperbolic quadrilateral in $Y$ with three angles of value $\frac{\pi}{2}$ that has one side on $\delta_1$ of length $\frac{\delta_1}{2}$, one side on $\alpha$ of length $\frac{\alpha'}{2}$ and two sides of infinite length meeting in the point at infinity defined by the cusp. It follows as for the quadrilateral $\mathcal{R}'$ that
\begin{equation}
\sinh(\frac{\delta_1}{2})^2  \geq \sinh(\frac{\delta_1}{2}) \sinh(\frac{\alpha'}{2}) = 1, \text{ \ hence \ } \frac{\delta_1}{2} \geq \arcsinh(1). 
\label{eq:delta1_alpha'}
\end{equation}
Now let $c_1$ and $c_2$ be two arcs on $\alpha$ that connect the endpoints of $\delta_1$ and $\delta_2$ on $\alpha$, such that
\[
 c_1 + c_2 \leq \frac{\alpha}{2}.
\]
The existence of these arcs follows from the pigeonhole principle. Let $\nu$ be the shortest simple closed geodesic in $X^1$ in the free homotopy class of $\delta_1 \cdot c_1 \cdot (\delta_2)^{-1} \cdot c_2$. As $\nu \geq \alpha$, we obtain from (\ref{eq:delta2_collar}), (\ref{eq:delta1_alpha'}) and the upper bound on $c_1+c_2$:
\begin{equation}
        \delta_1 +c_1 + \delta_2 + c_2 > \nu \geq \alpha  \text{ \ \ therefore \ \ } \frac{\delta_1}{2} \geq \max\{\arcsinh(1), \frac{\alpha}{4} - \arcsinh\left(\frac{1}{\sinh(\frac{\alpha}{4})}\right)\}.
\label{eq:delta1_collar_simp}        
\end{equation}
This is close to our desired result. Using more refined estimates for $\delta_1$ in $X^1$, one can show that indeed 
\begin{equation}
\frac{\delta_1}{2} \geq \max\{1.319,  \frac{\alpha}{4} - \arcsinh\left(\frac{1}{\sinh(\frac{\alpha}{4})}\right)\}.
\label{eq:delta1_lemma}
\end{equation}
To this end let $\mathcal{P}$ be a geodesic right-angled pentagon in $Y$ that has one side on $\delta_1$ of length $\frac{\delta_1}{2}$, one side on $\alpha$ of length $\frac{\alpha''}{2}$ and a side on $\eta$ of length $\frac{\eta}{2}$. It follows from the geometry of $\mathcal{P}$ that
\[
\sinh(\frac{\delta_1}{2}) \sinh(\frac{\alpha''}{2}) = \cosh(\frac{\eta}{2}).
\]
With the help of this equation, the equivalent relation in $\mathcal{R}$ from Equation (\ref{eq:delta1_alpha'}) and the fact that $\alpha' + \alpha'' = \alpha$, we find the following expression for $\delta_1$:
\begin{equation}
\sinh(\frac{\delta_1}{2}) = \sqrt{\frac{(\cosh(\frac{\alpha}{2}) + \cosh(\frac{\eta}{2}))^2}{\sinh(\frac{\alpha}{2})^2} -1}.
\label{eq:delta1_eta0}
\end{equation}
This can be deduced using the addition formulas for hyperbolic functions. As $\sinh(\frac{\delta_1}{2})$ is a monotonically increasing function with respect to $\eta$ and as $\eta \geq \alpha$, we obtain:
\begin{equation}
\sinh(\frac{\delta_1}{2}) \geq \sqrt{\frac{4\cosh(\frac{\alpha}{2})^2}{\sinh(\frac{\alpha}{2})^2} -1}.
\label{eq:delta1_eta}
\end{equation}
Hence in total we obtain (\ref{eq:delta1_lemma}) by combining (\ref{eq:delta1_collar_simp}) and (\ref{eq:delta1_eta}) using the fact that $\alpha \geq 3.525$ (see \cite{gs}). Combining inequality (\ref{eq:delta1_lemma}) with the inequality from \textit{Case 1} we obtain our lemma. 
\end{proof}
Now we show that a maximal surface $S_{max}$ that has a separating systole has another separating systole that intersects it. It follows with \textbf{Lemma \ref{thm:sys_coll_cusp}}:

\begin{lem} Let $S_{max}$ be a maximal non-compact hyperbolic surface of signature $(g,n) \neq (0,4)$, where $3g-3+n > 0$ and $n \geq 2$. If $S_{max}$ has a separating systole, then there exists a separating systole $\alpha$ that bounds a surface $Y'$ of signature $(0,3,2)$. The maximal half-collar $H_{\omega_{\alpha}}(\alpha)$ of $\alpha$ on the side opposite of $Y'$ has width
\[
      \omega_{\alpha} > \min\left\{\frac{\alpha}{4},\max\{1.319, \frac{\alpha}{4} - \arcsinh\left(\frac{1}{\sinh(\frac{\alpha}{4})}\right)\} \right\}.
\]
\label{thm:sep_sys_Smax}
\end{lem}
\begin{proof}
Let $S_{max}$ be a maximal hyperbolic surface of signature $(g,n)$ that has a separating systole $\alpha$ and whose signature satisfies the conditions of the lemma. Now $\alpha$ is intersected by another systole $\beta$. This follows from the fact that there is only a discrete number of maximal surfaces in the moduli space $\mathcal{M}_{g,n}$ of hyperbolic surfaces of signature $(g,n)$. Otherwise we could generate a smooth family of maximal surfaces by applying a small twist around $\alpha$, a contradiction. The remainder of \textbf{Lemma \ref{thm:sep_sys_Smax}} then follows from \textbf{Lemma \ref{thm:sys_coll_cusp}}.
\end{proof}

\textbf{proof of Theorem \ref{thm:intermsys2}}\\
\\
Let $S_{max}$ be a maximal surface of signature $(g,n)$ that has a separating systole and such that $(g,n)$ satisfies the conditions of the theorem. We cut open $S_{max}$ along the systole $\alpha$ that bounds a Y-piece $Y'$ of signature $(0,3,2)$ (see \textbf{Lemma \ref{thm:sep_sys_Smax}}). Let $R$ be the remainder of $S_{max}$ of signature $(g,n-1,n-2)$. 
Take two copies, $R^1$ and $R^2$, of $R$ and let $\alpha_1$ be the boundary of $R^1$ and $\alpha_2$ be the boundary of $R^2$.\\
\\
\textit{1. $\msy(2g,2n-4) > \min\left\{\msy(g,n),\max\{5.276,\msy(g,n)-4\arcsinh\left(\frac{1}{\sinh(\frac{\msy(g,n)}{4})}\right)\}\right\}$.}\\
\\
To obtain this inequality, we  paste $R^1$ and $R^2$  together with arbitrary twist parameter along the boundaries to obtain a surface $R'$ of signature $(2g,2n-4,2n-4)$. Due to the half-collar at $\alpha_1$ and $\alpha_2$ we obtain by similar arguments as in the proof of \textbf{Theorem \ref{thm:intermsys}} that  $\msy(2g,2n-4) > \sy(R') \geq \min\left\{\msy(g,n),\max\{5.276,\msy(g,n)-4\arcsinh\left(\frac{1}{\sinh(\frac{\msy(g,n)}{4})}\right)\}\right\}.$
\\
\textit{2. $\msy(2g+1,2n-4) > \frac{\msy(g,n)}{2}$.}\\
\\
As in the proof of \textbf{Theorem \ref{thm:intermsys}-2}, let $F_{max}$ be a surface of signature $(1,2,0)$ whose interior systole $\sy(F_{max})$ has maximal length among all F-pieces with two boundary geodesics of equal length $b$. Then
\[
   \sy(F_{max}) > \frac{b}{2}.
\]
Let $F'_{max}$ be such a surface where $\sy(F_{max}) = \msy(g,n) = \alpha$. We paste together $R^1,F'_{max}$ and $R^2$ along the boundaries to obtain a surface $R''$ of signature $(2g+1,2n-4,2n-4)$. It follows from the properties of this surface by similar arguments as in the proof of \textbf{Theorem \ref{thm:intermsys}-2} that $\msy(2g,2n-4) > \sy(R'') > \frac{\msy(2g+1,2n-4)}{2}$. \\
\\
\textit{3. $\msy(g+1,n-2) > \frac{\msy(g,n)}{3}$.}\\
\\
We call a surface of signature $(1,1,0)$ a \textit{Q-piece}. Let $b$ be the boundary geodesic of a Q-piece. Let $Q_{max}$ be a surface whose interior systole has maximal length among all Q-pieces with a boundary of length $b$. Denote by $\sy(Q_{max})$ the length of the interior systole of $Q_{max}$. It was shown in \cite{sc1} and \cite{pa3}, \textbf{Proposition 5.4} that
\[
    \cosh( \frac{\sy(Q_{max})}{2})=   \cosh( \frac{b}{6})+\frac{1}{2} \Rightarrow \sy(Q_{max}) > \frac{b}{3}.
\]
Let $Q'_{max}$ be such a surface with boundary length $\alpha$.
To prove the final part of the theorem, we paste together $R$ and $Q'_{max}$ along the boundaries to obtain a surface $R^*$ of signature $(g+1,n-2,n-2)$. It follows from the properties of this surface by similar arguments as in the proof of \textbf{Theorem \ref{thm:intermsys}} that  $\msy(g,2n-4) > \sy(R^*) > \frac{\msy(2g+1,2n-4)}{3}$. This concludes the proof of \textbf{Theorem \ref{thm:intermsys2}}. \hfill  $\square$ \\
\\
We now explain shortly, why no strong collar lemma should hold for non-separating systoles of non-compact hyperbolic Riemann surfaces. To this end we start by trying to prove such a collar lemma and then show where the arguments fail.\\ 
Let $S$ be a  hyperbolic surface of signature $(g,n)$ that has a non-separating systole $\alpha$. The closure $\overline{C_{\omega_{\alpha}}(\alpha)}$ of the maximal collar of $\alpha$ self-intersects in a point $p_1$. 
There exist two geodesic arcs $\delta'_1$ and $\delta''_1$ of length $\omega_{\alpha}$ emanating from $\alpha$ and perpendicular to $\alpha$ having the endpoint $p_1$ in common. These two arcs form a smooth geodesic arc $\delta_1$ (see Fig.~\ref{fig:X12}). The endpoints of $\delta_1$ on $\alpha$ divide $\alpha$ into two parts. We denote these two arcs by $\alpha'_1$ and $\alpha''_1$. Let without loss of generality $\alpha'_1$ be the shorter arc of these two. Then
\[
   \alpha'_1 \leq \frac{\alpha}{2}.
\]
Suppose that $\delta'_1$ and $\delta''_1$ emanate from the same side of $\alpha$.\\
Let $\gamma_1$ be the simple closed geodesic in the free homotopy class of $\alpha'_1\cdot \delta_1$ and suppose that $\gamma_1$ is of zero length, i.e. defines a cusp. Let $\eta_1$ be the curve in the free homotopy class of $\alpha''_1 \cdot \delta_1$. Then $\eta_1$ must be a non-separating curve, because otherwise $\alpha$ would be separating. We have that
\[
     \alpha \leq \eta_1 < \alpha''_1 + \delta_1 \leq \alpha + 2 \omega_{\alpha}.
\]
Let $Y_2$ be the surface of signature $(0,3,1)$ defined by the boundary geodesics $\eta_1$, $\alpha$ and the cusp defined by $\gamma_1$. Now we assume that $\alpha=\eta_1$. In this case it can be shown from the hyperbolic geometry of the Y-piece $Y_2$ (see Equation (\ref{eq:delta1_eta0}) and (\ref{eq:delta1_eta}) replacing $\eta$ by $\eta_1$) that 
\begin{equation}
\delta_1 =  2 \omega_{\alpha} \leq 2\arcsinh(\sqrt{3})+ \epsilon,
\label{eq:small_collar}
\end{equation}
where $\epsilon > 0$ and $\epsilon \to 0$ if $\alpha \to \infty$.\\
Now we might like to try to show the existence of a large half-collar on the opposite side of $\alpha$. However, we can not exclude that the situation is symmetric. More precisely, suppose that there is a Y-piece $Y'_2$ isometric to $Y_2$ embedded in $S$ and that $Y_2$ and $Y'_2$ have common boundary $\alpha$. Let $\psi: Y_2 \rightarrow Y'_2$ be the isometry between these two surfaces and set
\[ 
  \psi(\delta'_1) = \delta'_2,\,\,\psi(\delta''_1) = \delta''_2,\,\, \psi(\delta_1) = \delta_2,\,\, \psi(\alpha'_1) = \alpha'_2 \text{ \ etc. \ }
\]
In this case the closure $\overline{C_{\omega_{\alpha}}(\alpha)}$ of the maximal collar of $\alpha$ on the opposite side of $Y'_2$ also self-intersects in a point $p_2 \in Y_2$, such that the two geodesic arcs of length $\omega_{\alpha}$ emanating from $\alpha$ and perpendicular to $\alpha$ having the endpoint $p_2$ in common are $\delta'_2$ and $\delta''_2$.\\ 
Now let $c_1$ and $c_2$ be two arcs on $\alpha$ that connect the endpoints of $\delta_1$ and $\delta_2$ on $\alpha$, such that $c_1 + c_2$ is minimal. Let $\nu$ be the curve in the free homotopy class of $\delta_1 \cdot c_1 \cdot (\delta_2)^{-1} \cdot c_2$. Now 
\[
             \alpha  \leq \nu.
\]
One would like to argue that $\omega_\alpha$ is in fact large due to this inequality, which would lead to a contradiction to inequality (\ref{eq:small_collar}).\\
But then again, if the twist parameter in the pasting of $Y + Y'$ is $\frac{1}{2}$ i.e for any point $q$ on $\alpha$, $\dist(q,\psi(q))=\frac{\alpha}{2}$ these arguments fail. This can be shown by calculating the length of $\nu$ explicitly. However, for the sake of brevity, we omit the details.\\
Though this example does not provide a rigorous proof that a non-separating systole of a non-compact maximal surface does not have a large collar, it leads us to conjecture that this can indeed occur.\\  
\\
We now prove the following lemma, from which follows \textbf{Theorem \ref{thm:intermsys3}} by passing to a maximal surface. 
\begin{lem} Let $S$ be a non-compact hyperbolic Riemann surface of signature $(g,n)$, where $3g-3 + n > 0$ and $g\geq 1$. There is a normal cover $\tilde{S}$ of order $k$ and of signature $(k(g-1)+1, k n)$, such that 
\[
\sy(\tilde{S}) \geq \sy(S).
\]
\label{thm:kcover_lem2}
\end{lem}

\begin{proof}
Let $S$ be a hyperbolic Riemann surface of signature $(g,n)$, where $3g-3 + n > 0$ and $g\geq 1$. To prove the theorem we construct a new surface $\tilde{S}$ of signature $(k(g-1)+1,kn)$ from $S$ such that
\[
\sy(\tilde{S}) \geq \sy(S).
\]
Let $\alpha$ be a non-separating simple closed geodesic of $S$. We use the same construction from the proof of \textbf{Note \ref{thm:kcover_note}} to construct a covering surface $\tilde{S}$ by cutting open $S$ along $\alpha$ and pasting $k$ copies of the cut surface as described there. Recall that $\tilde{\phi}$ is the isometry, such that 
\[
S = \tilde{S} \mod \tilde{\phi} \text{ \ \  and let \ \ } p: \tilde{S} \rightarrow S
\]
be the corresponding covering map which is a local isometry. Now, to prove the intersystolic inequality between $\sy(\tilde{S})$ and $\sy(S)$ we can not use the argument with the injectivity radius. In the non-compact case with cusps we have that $r_{inj}(S) = 0$, whereas $\sy(S) > 0$.\\ 
Therefore we will prove the inequality from \textbf{Lemma \ref{thm:kcover_lem2}} in another way. To this end we remark the following:
\begin{itemize}  
\item A systole in $S$ or $\tilde{S}$ is a simple closed geodesic.
\item For any simple closed geodesic $\tilde{\gamma} \in \tilde{S}$, we have that $p(\tilde{\gamma})=\gamma$  is a smooth closed geodesic. The length of $\gamma$ is smaller or equal to the length of $\tilde{\gamma}$, but $\gamma$ has always non-zero length.   
\end{itemize}
Now let $\tilde{\alpha}$ be a systole of $\tilde{S}$ of length $\sy(\tilde{S})$. Then $p(\tilde{\alpha})=\alpha'$ is a smooth closed geodesic in $S$ of non-zero length. As it is smooth it is the shortest curve in its free homotopy class. Now if $\alpha'$ defines a cusp, it would follow that $\alpha'$ has zero length. A contradiction. Hence $\alpha'$ has non-zero length and it follows that
\[
   \sy(\tilde{S}) = \tilde{\alpha} \geq \alpha' \geq \sy(S).
\]
This concludes our proof. 
\end{proof}

\noindent Hugo Akrout\\	
\noindent Department of Mathematics, Universit\'e Montpellier 2 \\
\noindent place Eug\`ene Bataillon, 34095 Montpellier cedex 5, France \\
\noindent e-mail: \textit{akrout@math.univ-montp2.fr} \\
\\
\noindent Bjoern Muetzel\\
\noindent Institute for Algebra and Geometry, Karlsruhe Institute of Technology\\
\noindent Kaiserstrasse 89-93, office 4B-03 (Allianz building)\\
\noindent 76133 Karlsruhe, Germany\\
\noindent e-mail: \textit{bjorn.mutzel@gmail.com}

\end{document}